\documentclass{article}
\def\dj{d\kern-0.4em\char"16\kern-0.1em}
\def\Dj{\mbox{\raise0.3ex\hbox{-}\kern-0.4em D}}

\usepackage{amsfonts}
\usepackage{amssymb}
\usepackage{amsthm}
\usepackage{amsmath}
\usepackage{newlfont}
\usepackage{graphicx}
\usepackage[all,arc]{xy}
\usepackage{epsfig}

\usepackage{multicol}

\newtheorem {proposition}{Proposition}[section]
\newtheorem {theorem}{Theorem}[section]
\newtheorem {lemma}{Lemma}[section]
\newtheorem {example}{Example}[section]

\newtheorem {conjecture}{Conjecture}

\newtheorem {corollary}{Corollary}[section]

\newcommand\imm{\operatorname{imm}}
\newcommand\rank{\operatorname{rank}}
\newcommand\Em{\operatorname{em}}
\newcommand\w{\overline{w}}

\newcommand\bn{\tbinom}

\author{{\Dj{}or\dj{}e Barali\' {c}}\\ {\small Mathematical Institute of Serbian Academy of Sciences and Arts}\\[-2mm] {\small Belgrade, Serbia}
\\[-2mm]{\small djbaralic@mi.sanu.ac.rs}
\and Vladimir Gruji\'{c}\\ {\small Faculty of Mathematcs}\\[-2mm] {\small University of Belgrade,
Serbia}
\\[-2mm]{\small vgrujic@matf.bg.ac.rs}}

\title{Quasitoric manifolds and Small covers over properly colored polytopes: Immersions and Embeddings}
\date{}
\begin{document}
\maketitle

\begin{abstract} We construct small covers and quasitoric
manifolds over $n$-dimensional simple polytopes which allow proper
colorings of facets with $n$ colors. We calculate Stiefel-Whitney
classes of these manifolds as obstructions to immersions and
embeddings into Euclidean spaces. The largest dimension required
for embedding is achieved in the case $n$ is a power of two.

\hspace{0 mm} \textbf{2010 Mathematics Subject Classification}:
    57N35,     52B12.

\hspace{0 mm} \textbf{Keywords}: immersions, quasitoric manifolds,
simple polytopes, colorings, Stiefel-Whitney classes.

\end{abstract}

\renewcommand{\thefootnote}{}
\footnotetext{This work was supported by Ministry of Education,
Science and Technological developments of Republic of Serbia
(Grants 174020 and 174034) and by Ministry of Science and
Technology of Republic of Serbska, Bosnia and Herzegovina (Grant
19/6-020/961-120/14).}

\section{Introduction}

\subsection{Colorings of simple polytopes}

An $n$-dimensional convex polytope is \textit{simple} if the
number of facets which are meeting at each vertex is equal to $n$.
The \textit{proper coloring} into $k$ colors of a simple polytope
$P^n$ is a map $$h: \{F_1, \dots, F_m\}\rightarrow
\{1,\ldots,k\}$$ of its set of facets such that $h (F_i)\neq h
(F_j)$ for each two intersecting facets. The \textit{chromatic
number} $\chi (P^n)$ of a simple polytope $P^{n}$ is the least $k$
for which there exists a proper coloring of $P^n$ into $k$ colors.
The chromatic numbers of the simplex $\Delta^n$, the cube $I^n$
and the permutahedron $\Pi^n$ (Figure \ref{bojenja}) are
$$\chi (\Delta^n)=n+1, \ \ \ \chi (I^n)=n, \ \ \ \chi
(\Pi^n)=n.$$

\begin{figure}[h!h!]
\centerline{\includegraphics[width=\textwidth]{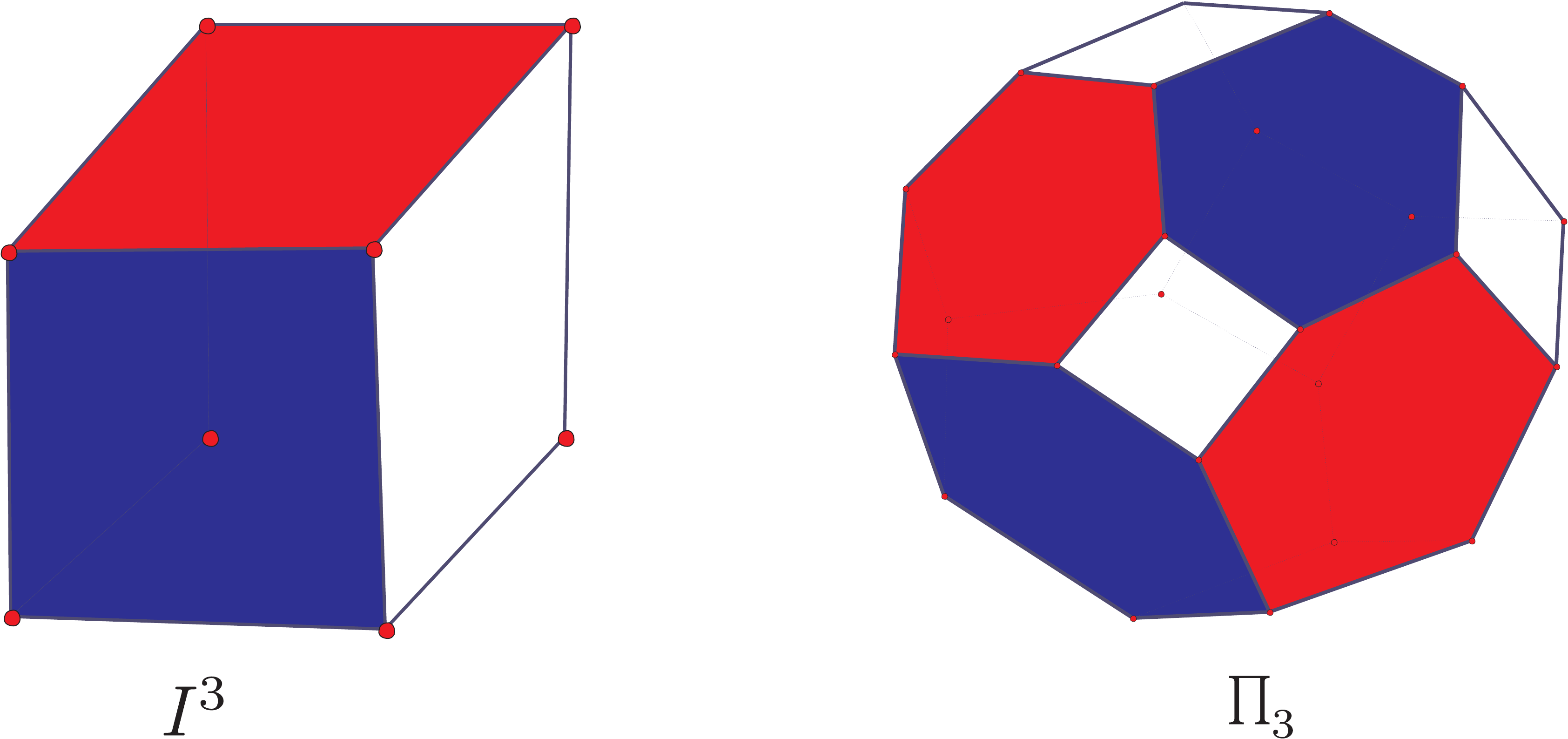}}
\caption{The colorings of the cube and the permutahedron}
\label{bojenja}
\end{figure}

Obviously, $\chi (P^n)\geq n$ for any simple polytope $P^n$. The
chromatic number of a polygon is clearly equal to $2$ or $3$,
depending on the parity of the number of its faces. By famous Four
Color Theorem chromatic numbers of $3$-dimensional simple
polytopes are equal to $3$ or $4$. In general case, for $n\geq 4$
it does not hold $\chi (P^n)\leq n+1$. Moreover, there are simple
polytopes such that their chromatic numbers are equal exactly to
numbers of their facets. The polars of the cyclic polytopes with
$m$ vertices $C^n (m)$ are examples of such type (see
\cite[Example~0.6]{5}).

We consider the class of $n$-dimensional simple polytopes with the
chromatic number equal to $n$. Denote this class by $\mathcal{C}$.
The class $\mathcal{C}$ is closed with respect to products (see
\cite[Construction~1.12]{2}) and connected sums (see
\cite[Construction~1.13]{2}). Also from any given simple polytope
$P^n$ by truncation over all its faces we obtain a simple polytope
$Q^n$ which belongs to the class $\mathcal{C}$. The complete
description of the class $\mathcal{C}$ is obtained in \cite{9}. A
simple $n$-polytope $P^n$ admits a proper coloring in $n$ colors
if and only if every its $2$-face has an even number of edges (see
\cite[Theorem 16]{9}).

\subsection{Small covers and quasitoric manifolds}

Quasitoric manifolds and they real analogues called small covers
are introduced by Davis and Januszkiewicz in \cite{3}. Their
geometric and algebraic topological properties are closely related
to combinatorics of simple polytopes. The following definitions
and notions are extensively elaborated in \cite{2}. Denote by
$$G_d=\left\{\begin{array}{rl}
S^{0}, & d=1,\\
S^{1},  & d=2,\end{array} \right. \ R_d=\left\{\begin{array}{rl}
\mathbb{Z}_2, & d=1,\\
\mathbb{Z},  & d=2,\end{array}\right. \
\mathbb{K}_d=\left\{\begin{array}{rl}
\mathbb{R}, & d=1,\\
\mathbb{C},  & d=2,\end{array} \right.$$ where $S^{0}=\{-1,+1\}$
and $S^{1}=\{z\mid |z|=1\}$ are multiplicative subgroups of real
and complex numbers, $\mathbb{Z}$ is the ring of integers and
$\mathbb{Z}_2=\{0,1\}$ is the ring of integers modulo $2$.

The group $G_d^{n}$ acts standardly on $\mathbb{K}_d^{n}$ by
$(t_1,\ldots,t_n)\cdot(x_1,\ldots,x_n)=(t_1x_1,\ldots,t_nx_n)$.
The action of $G_d^{n}$ on a smooth $dn$-dimensional manifold
$M^{dn}$ is locally standard if for any point of $M^{dn}$ there is
a $G_d^{n}$-invariant neighborhood which is weakly equivariantly
diffeomorphic to some open $G_d^{n}$-invariant subset of
$\mathbb{K}_d^{n}$ with the standard action of $G_d^{n}$. Recall
that two $G_d^{n}$-manifolds are weakly equivariantly
diffeomorphic if there is an automorphism $\omega:G_d^n\rightarrow
G_d^n$ and a diffeomorphism $f:M_1^{dn}\rightarrow M_2^{dn}$ such
that $f(g\cdot x)=\omega(g)\cdot f(x)$ for any $g\in G_d^n$ and
$x\in M_1^{dn}$. A smooth $dn$-dimensional $G_d^{n}$-manifold
$M^{dn}$ with a locally standard action of the group $G_d^{n}$
such that the orbit space $M^{dn}/G_d^{n}$ is diffeomorphic, as a
manifold with corners, to a simple $n$-dimensional polytope
$P^{n}$ is called {\it small cover} over $P^{n}$ if $d=1$,
correspondingly {\it quasitoric manifolds} over $P^{n}$ if $d=2$.
Let $\pi:M^{dn}\rightarrow P^{n}$ be the projection map and
$\{F_1,\ldots,F_m\}$ be the set of facets of the polytope $P^{n}$.
The inverse images $M^{d(n-1)}_j=\pi^{-1}(F_j)$ are
$G_d^{n}$-invariant submanifolds of codimension $d$ called {\it
characteristic submanifolds}. To each characteristic submanifold
$M_j^{d(n-1)}$ corresponds the isotropy subgroup
$G(F_j)=G_{\lambda_j}$ of the rank one, where
$G_{\lambda_j}=\{(1,\ldots,1),((-1)^{\lambda_{1j}},\ldots,(-1)^{\lambda_{nj}})\}$
for some
$\lambda_j=(\lambda_{1j},\ldots,\lambda_{nj})\in\mathbb{Z}_2^{n}\setminus\{0\}$
if $d=1$ and $G_{\lambda_j}=\{(e^{2\pi
i\lambda_{1j}t},\ldots,e^{2\pi i\lambda_{nj}t})\mid
t\in\mathbb{R}\}$ for some primitive vector
$\lambda_j=(\lambda_{1j},\ldots,\lambda_{nj})\in\mathbb{Z}^{n}$
defined up to the sign, if $d=2$. In this way, the action of the
group $G_d^{n}$ on $M^{dn}$ defines the {\it characteristic map}
$$l_d:\{F_1, \dots, F_m\}\rightarrow R^n_d,$$ which to a facet
$F_j$ of the polytope $P^{n}$ assigns a primitive vector
$\lambda_j=(\lambda_{1j},\ldots,\lambda_{nj})\in R^n_d$. Denote by
$\Lambda= (\lambda_1, \dots, \lambda_m)$ the matrix formed by
these vectors. Then $\det \Lambda_{{(V)}}=\pm 1$ for any vertex
$V\in P^{n}$, where $\Lambda_{(V)}$ is the submatrix of the matrix
$\Lambda$ formed by vectors corresponding to facets which are
intersecting in that vertex. A pair $(P^{n}, \Lambda)$ satisfying
this condition on submatrices is called a {\it characteristic
pair}. The manifold $M^{dn}$ is reconstructible from the
characteristic pair $(P^{n}, \Lambda)$ up to weak equivariant
diffeomorphism in the following way (see \cite{3} and
\cite[Construction~6.18]{2}). Let $F_{i_1}\cap\ldots\cap F_{i_k}$
be the minimal face containing a point $q\in P^{n}$. To the point
$q$ we associate the subgroup
$G(q)=G_{\lambda_{i_1}}\times\cdots\times G_{\lambda_{i_k}}$.
Define $M(P^{n}, \Lambda)=G_d^{n}\times P^{n}/\sim$, where
$(t_1,p)\sim(t_2,q)$ if and only if $p=q$ and $t_1t_2^{-1}\in
G(q)$.

Though not all simple polytopes allow characteristic maps (see
\cite[Example~1.15]{3}), to considered polytopes
$P^{n}\in\mathcal{C}$ in a simple way can be assigned a
characteristic matrix $\Lambda$ if each color $i\in\{1,\ldots,n\}$
is identified with the vector $e_i$ of the standard basis in
$R_d^{n}$.

The cohomology ring of a small cover and a quasitoric manifold
$M^{dn}, d=1,2$ is described in the following way. Let
$v_j=D[M_j^{d(n-1)}], j=1,\ldots,m$ be cohomology classes which
are Poincare duals to fundamental classes of characteristic
manifolds. The characteristic matrix $\Lambda=(\lambda_1, \dots,
\lambda_m)$ defines following linear forms
$$\theta_i:=\sum_{j=1}^m \lambda_{i j} v_j, i=1,\ldots,n,$$ where
$\lambda_j=(\lambda_{1 j}, \dots, \lambda_{n j})^t \in R_d^n,
j=1,\ldots,m$. Let $\mathcal{J}$ be the ideal in $R_d[v_1, \dots,
v_m]$ generated by elements $\theta_1,\ldots,\theta_n$ and
$\mathcal{I}$ be the Stanley-Reisner ideal of the polytope
$P^{n}$, which is generated by monomials $v_{i_1}\cdots v_{i_k}$
whenever $F_{i_1}\cap\ldots F_{i_k}=\emptyset$ in $P^{n},
i_1<\ldots<i_k$. Then for the cohomology ring the following
isomorphism holds (see \cite{3})

\begin{equation} \label{djf} H^\ast (M^{dn},R_d)\simeq R_d[v_1, \dots,
v_m]/(\mathcal{I}+\mathcal{J}).\end{equation} The same formula
$(\ref{djf})$ holds with $\mathbb{Z}_2$ coefficients if $d=2$. The
total Stiefel-Whitney class is determined by the following
Davis-Januszkiewicz's formula

\begin{equation}
\label{djfswc} w (M^{dn})= \prod_{i=1}^m (1+v_i)\in H^\ast
(M^{dn}; \mathbb{Z}_2),\end{equation} where in the case $d=2$ all
classes $v_i$ are regarded as $\mathbb{Z}_2$-restrictions of
corresponding integral classes.

\subsection{Immersions and Embeddings}

Immersions and embeddings of manifolds are a classical topic in
algebraic topology and differential topology. We consider
immersions and embeddings in the smooth category. For a manifold
$M^{n}$ define the numbers $imm (M^n)$ and $em(M^n)$ as the
smallest dimensions of Euclidean spaces in which this manifold is
immersed and embedded, correspondingly. In accordance with
Whitney's theorem for each smooth manifold $M^{n}$ is satisfied
$imm(M^n)\leq 2n-1$ and $em(M^n)\leq 2n$. On the other hand, the
dual Stiefel-Whitney classes $\overline{w}_i$ serve as
obstructions to immersions and embeddings of manifolds in
Euclidean spaces. Recall that dual Stiefel-Whitney classes
$\overline{w}_i(M^{n})$ are characteristic classes of the stable
normal bundle of a manifold $M^{n}$.

\begin{theorem}[see \cite{4}]\label{imeem} For a smooth manifold $M^{n}$ let $k:=\max \{i\mid
\w_i(M^n)\neq 0\}$. Then $\imm (M^n)\geq n+k \, \mbox{\,and\,}\,
\Em (M^n)\geq n+k+1.$
\end{theorem}

The Stiefel-Whitney classes are also obstructions to so called
totally skew embeddings introduced by Ghomi and Tabachnikov in
\cite{6}. A manifold $M^{n}$ is totally skew embedded in the
Euclidean space $\mathbb{R}^{N}$ if any two tangent lines in
different points of $M^{n}$ are skew lines in $\mathbb{R}^{N}$. It
is proved in the paper \cite{6} that the number $N(M^{n})$ defined
as the smallest dimension of Euclidean spaces for which there is a
totally skew embedding of $M^{n}$, satisfies
\[2n+2\leq N(M^n)\leq 4n+1.\] The better lower bound is obtained
in \cite{7}.
\begin{theorem}[see \cite{7}, Proposition~1, Corollary~4]\label{skewteo}
If $$k=\max \{i\mid \overline{w}_i(M^n)\neq 0\}$$ then $N(M^n)\geq
2n+2k+1.$
\end{theorem}

The immersions and embeddings of quasitoric manifolds over cubes
are studied in \cite{1}.

In this paper we prove the following results.

\begin{theorem}\label{main} Let $n$ be a power of two and $P^n$ be a simple convex
$n$-dimensional polytope which allows a proper coloring in $n$
colors.
\begin{itemize}
\item[(1)] There is a small cover $M^{n}$ over the polytope
$P^{n}$ which satisfies following identities $imm(M^{n})=2n-1$ and
$em(M^{n})=2n$.

\item[(2)] There is a quasitoric manifold $M^{2n}$ over the
polytope $P^{n}$ which satisfies following inequalities
$imm(M^{2n})\geq 4n-2$ and $em(M^{2n})\geq 4n-1$. Moreover, for
$n\geq 3$ in both relation the equality holds.
\end{itemize}
\end{theorem}

We give an explicit construction of manifolds $M^{n}$ and $M^{2n}$
from Theorem \ref{main}. By using combinatorial properties of the
simple polytope $P^n$ and Davis-Januszkiewicz formula (\ref{djf})
and (\ref{djfswc}), we describe the cohomology rings $H^*(M^{dn},
\mathbb{Z}_2), d=1,2$ and calculate Stiefel-Whitney classes
$w_k(M^{dn}), d=1,2$. When $n$ is a power of two, we prove that
classes $\w_{d(n-1)}( M^{dn}), d=1,2$ are nontrivial, the claim
that implies Theorem \ref{main}. Also from this claim and Theorem
\ref{skewteo} for totally skew embeddings of constructed manifolds
immediately follows

\begin{corollary} If $n$ is a power of two then
$$ N(M^{n})\geq 4n-1, \ \ \ N(M^{2n})\geq 8n-3.$$
\end{corollary}

The obtained small covers $M^{n}$ are a new class of manifolds for
which the numbers $N(M^{n})$ are equal to $4n-1, 4n,$ or $4n+1$.
So far, only known examples with this property were real
projective spaces (see \cite{7}).

Cohen \cite{8} in 1985 resolved positively the famous {\em
Immersion Conjecture}, by showing that each compact smooth
$n$-manifold for $n>1$ can be immersed in
$\mathbb{R}^{2n-\alpha(n)}$, where $\alpha(n)$ is the number of
1's in the binary expansion of $n$. The products of real
projective spaces are examples of manifolds that achieve the upper
bounds. We construct new examples of this type in the class of
small covers.

\begin{theorem}\label{exist}
For every positive integer $n$ there is a small cover $M^{n}$ and
a quasitoric manifold $M^{2n}$ over some simple $n$-dimensional
polytope which allows a proper coloring in $n$ colors such that
\begin{eqnarray*}
imm(M^{n})=2n-\alpha(n), \ \ \ imm(M^{2n})\geq 4n-2\alpha(n) \\
em(M^{n})=2n-\alpha(n)+1, \ \ \ em(M^{2n})\geq 4n-2\alpha(n)+1, \\
N(M^{n})\geq 4n-2\alpha(n)+1, \ \ \ N(M^{n})\geq 8n-4\alpha(n)+1.
\end{eqnarray*}
\end{theorem}

We conjecture that the statement of Theorem \ref{exist} is
satisfied for any simple $n$-dimensional polytope which is
properly colored by $n$ colors.

\begin{conjecture}\label{conj1}
Let $P^n$ be a simple convex $n$-dimensional polytope which allows
a proper coloring in $n$ colors. Then there exist a small cover
$M^{n}$ and a quasitoric manifold $M^{2n}$ over the polytope $P^n$
which satisfy

\begin{eqnarray*}
imm(M^{n})= 2n-\alpha(n),\;\;\;\;em(M^{n})= 2n-\alpha(n)+1,\\
imm(M^{2n})\geq  4n-2\alpha(n),\;\;\;\;em(M^{2n})\geq 4n-2
\alpha(n)+1.
\end{eqnarray*}
\end{conjecture}

\section{Manifolds $M^{dn}$}

\subsection{Construction}\label{cons}

Let $P^n$ be a simple polytope such that $\chi (P^n)=n$ and
$h:\{F_1, \dots,F_m\}\rightarrow\{1,\ldots,n\}$ be its proper
coloring. Denote by $\mathcal{F}_j$ the set $h^{-1} (j)$. Every
vertex $V\in P^n$ is the intersection of $n$ differently colored
facets. Take an arbitrary vertex $V=H_1\cap\cdots\cap H_n$, where
the facet $H_i$ is colored by color $i$. Assign to each facet
$H_i$ the vector
$\lambda_i=(\underbrace{0,\dots,0}_{i-1},\underbrace{1,\dots,1}_{n-i+1})^t$
and vectors $\lambda_F=(\underbrace{0,\dots,0}_{i-1}, 1,
\underbrace{0,\dots,0}_{n-i})^t$ to remaining facets $F\in
\mathcal{F}_i\backslash \{H_i\}$. The corresponding matrix
$\Lambda$ clearly induces the characteristic map, since $\det
{\Lambda}_{(v)}=1$ for every vertex $V\in P^n$. Define
$M^{dn}=M(P^{n}, \Lambda)$ as the manifold constructed from the
characteristic pair $(P^{n}, \Lambda)$.

\subsection{Cohomology ring}

Let $u_1,\dots,u_n$ be Poincar\'e duals to characteristic
submanifolds over the facets $H_1,\dots,H_n$ respectively. For
every facet $F$ of $P^n$ distinct from $H_1,\dots,H_n$, let $v_F$
denotes the Poincar\'e dual to the characteristic submanifold over
$F$.

The cohomology ring of the manifold $M^{dn}$ is determined by the
Stanley-Reisner ideal $\mathcal{I}$ of $P^n$ and the ideal
$\mathcal{J}$ which is generated by linear forms
\begin{equation} \label{j1}
\theta_i:=\sum_{j=1}^i u_j+\sum_{F\in
{\mathcal{F}_i\setminus\{H_i\}}} v_F, i=1, \dots, n.
\end{equation}
Recall that $$H^*(M^{dn}; \mathbb{Z}_2)=\mathbb{Z}_2[u_1, \dots,
u_n, v_F |F\in {\mathcal{F}\setminus\{H_1,\dots,
H_n\}}]/(\mathcal{I}+\mathcal{J}).$$

From the coloring of $P^n$ we easily deduce

\begin{proposition}\label{p1} For facets $F$, $G\in {\mathcal{F}_i\setminus\{H_i\}}$, $F\neq G$
corresponding classes satisfy the following relation in the
cohomology ring $H^*(M^{dn}; \mathbb{Z}_2)$
\[v_F v_G=u_i v_F=u_i v_G=0.\]
\end{proposition}

Proposition \ref{p1} and (\ref{j1}) together imply

\begin{proposition}\label{p2} The following equalities hold in $H^*(M^{dn};
\mathbb{Z}_2)$
\begin{equation}\label{r1}
u_1^2=0, \, u_2^2=u_1 u_2, \dots, u_n^2=u_1 u_n+\cdots+u_{n-1}
u_n.
\end{equation}
\end{proposition}

Let $F$ be an arbitrary facet of $P^n$ and $v_F$ the corresponding
class over $F$, even if $F$ is one of $H_1$, $\dots$, $H_n$.

\begin{lemma}\label{lema1}
Let $k$ be a positive integer and $h$ a proper coloring map of the
polytope $P^n$. Then the class $v_F^k$ is either trivial or equal
to the degree $k$ homogenous polynomial
$Q_k^F(u_1,\ldots,u_{h(F)-1},v_F)$ whose monomials are square
free.
\end{lemma}

\begin{proof} Let $h(F)=i$. We prove the claim by induction on $k$. For
$k=1$ it is trivial. Assume that $v_F^k=Q_k^F (u_1,\ldots,u_{i-1},
v_F)$. By multiplying with $v_F^k$ the relation
$$\sum_{j=1}^{i-1} u_j + \sum_{F\in
{\mathcal{F}_i\setminus\{H_i\}}} v_{F}=0$$ and by using
Proposition \ref{p1} we get
\[v_F^{k+1}=(u_1+\cdots+u_{i-1}) Q_k^F (u_1, \dots, u_{i-1},
v_F).\] If $v_F^{k+1}=0,$ the claim follows directly. In the
opposite case it is obvious that $(u_1+\cdots+u_{i-1}) Q_k^F (u_1,
\dots, u_{i-1}, v_F)$ is the degree $k+1$ homogenous polynomial in
variables $u_1$, $\dots$, $u_{i-1}$, $v_F$. However, from
relations (\ref{r1}) follows $(u_1+\cdots+u_{i-1}) Q_k^F (u_1,
\dots, u_{i-1}, v_F)=Q_{k+1}^F (u_1, \dots, u_{i-1}, v_F)$, where
all monomials of the polynomial $Q_{k+1}^{F}$ are square free.
\end{proof}

From lemma \ref{lema1} immediately follows

\begin{corollary}\label{posl1}
Each class  $v_{F_{i_1}}^{r_1}\cdots v_{F_{i_k}}^{r_k}$ is either
trivial or equal to the degree $r_1+\dots+r_k$ homogenous
polynomial in variables $u_1$, $\dots$, $u_{n}$, $v_{F_{i_1}}$,
$\dots$,  $v_{F_{i_k}}$ whose monomials are square free.
\end{corollary}

\begin{proposition}\label{p3} The class $u_1 \dots u_n$ is the fundamental
cohomology class in the ring $H^{dn} (M^{dn}; \mathbb{Z}_2)$.
\end{proposition}

\begin{proof} Denote by $\mathcal{F}_V$ the set of facets containing a vertex $V\in
P^{n}$. Let $V^{\ast}=\prod_{F\in\mathcal{F}_V}v_F$. We show that
$V^{\ast}=V'^{\ast}$ for each two vertices $V,V'\in P^{n}$. For
this it is sufficient to suppose that vertices are connected by
the edge $VV'$. In this case we have
$\mathcal{F}_{V'}=\mathcal{F}_V\setminus\{G\}\cup\{G'\}$ for
unique facets $G$ and $G'$ which are colored by the same color. If
$h(G)=h(G')=i$, we multiply by
$\prod_{F\in\mathcal{F}_V\cap\mathcal{F}_{V'}}v_F$ the relation
$$\sum_{j=1}^{i-1} u_j + \sum_{F\in \mathcal{F}_i} v_{F}=0.$$
From Proposition \ref{p1} we obtain
$$(u_1+\ldots+u_{i-1})\prod_{F\in\mathcal{F}_V\cap\mathcal{F}_{V'}}v_F+V^{\ast}+V'^{\ast}=0.$$
In the case $i=1$ we get the required equality
$V^{\ast}=V'^{\ast}$. If $i>1$, for each $1\leq j<i$, we have by
Propositions \ref{p1} and \ref{p2}
$$u_j\prod_{F\in\mathcal{F}_V\cap\mathcal{F}_{V'}}v_F=\left\{\begin{array}{cc} 0,
&\mbox{if} \ H_j\notin\mathcal{F}_V\cap\mathcal{F}_{V'} \\
(u_1+\cdots
u_{j-1})\prod_{F\in\mathcal{F}_V\cap\mathcal{F}_{V'}}v_F, &
\mbox{if} \
H_j\in\mathcal{F}_V\cap\mathcal{F}_{V'}\end{array}\right..$$ Then
by induction follows
$u_j\prod_{F\in\mathcal{F}_V\cap\mathcal{F}_{V'}}v_F=0,
j=1,\ldots,i-1$ which again gives $V^{\ast}=V'^{\ast}$.

Suppose now that $u_1\cdots u_n=0$. From the so far proven part
follows that $V^{\ast}=0$ for each vertex $V\in P^{n}$. Moreover
in this case all $dn$-dimensional classes vanish since by
Corollary \ref{posl1} square free monomials of degree $n$ linearly
generate $H^{dn} (M^{dn}; \mathbb{Z}_2)$. But it contradicts the
known fact that $H^{dn} (M^{dn}; \mathbb{Z}_2)\simeq\mathbb{Z}_2$.
\end{proof}

\subsection{Stiefel-Whitney classes}

Proposition \ref{p1} implies that the total Stiefel-Whitney class
could be expressed as
\[w(M^{dn})=\prod_{i=1}^{n} \prod_{F\in \mathcal{F}_i} (1+v_F)=\prod_{i=1}^{n}(1+\sum_{F\in \mathcal{F}_i}v_F).\]
By applying (\ref{j1}) we obtain

\[w (M^{dn})=(1+u_1)(1+u_1+u_2)\cdots(1+u_1+u_2+\cdots+u_{n-1}).\]

In order to prove the main theorem \ref{main}, we are going to use
another set of generators $t_1,\dots,t_n$ which are defined by
\begin{eqnarray*}\label{tgen}
t_i=\sum_{j=1}^{i}u_j, i=1,\ldots,n.
\end{eqnarray*}
Consequently we have
\begin{equation}\label{sw3} w(M^{dn})=(1+t_1)\cdots(1+t_{n-1}).
\end{equation}
By Proposition \ref{p2} the classes $t_1^2$, $t_2^2+t_1
t_2,\ldots,t_n^2+t_{n-1}t_n$ vanish. Let $\mathcal{T}_n$ be the
ideal generated by these classes. By Proposition \ref{p3} it is
easily seen that the class $t_1 t_2 \cdots t_n$ also represents
the fundamental class.

\section{Proof of Theorem \ref{main}}

In order to prove Theorem \ref{main} it is sufficient to show that
the top dual Stiefel-Whitney class $\overline{w}_{d(n-1)}(M^{dn})$
is nontrivial.

\subsection{Dual Stiefel-Whitney classes}

Stiefel-Whitney classes and dual Stiefel-Whitney classes are
related by
$$w(M^{dn})\cdot \overline{w}(M^{dn})=1.$$
From the relation (\ref{sw3}) we obtain

\begin{lemma}\label{swd}
The total Stiefel--Whitney class $\overline{w}(M^{dn})$ is
expressed by
$$\overline{w}
(M^{dn})=(1+t_1)(1+t_2+t_2^2)\cdots(1+t_{n-1}+\cdots+t_{n-1}^{n-1}).$$
\end{lemma}

\begin{proof}
The statement follows from the fact that $t_k^{k+1}=0$ for all
$k=1,\ldots,n-1$. In fact $t_k^{k+1}=(u_1+\cdots+u_k)^{k+1}$ is
the sum of monomials of the form $\sum u_{i_1}^{r_1}\cdots
u_{i_j}^{r_j},$ where $1\leq i_1<\cdots<i_j\leq k$ and $r_1+\cdots
r_k=k+1$. In its turn, by Corollary \ref{posl1} and Proposition
\ref{p2} this is a sum of homogeneous polnomials of the degree
$k+1$ of the form $\sum Q_{k+1}(u_1,\ldots,u_k)$. Since monomials
of each polynomial in the sum is square free, we have
$Q_{k+1}(u_1,\ldots,u_k)=0$.
\end{proof}

We want to determine the highest nontrivial dual $\overline{w}_k
(M^{dn})$. For small $n$, we could calculate $\overline{w}
(M^{dn})$ directly.

\begin{example}\begin{itemize}
\item[(1)] $\overline{w} (M^{d \cdot 2})=1+t_1$, \item[(2)]
$\overline{w} (M^{d\cdot 3})=1+(t_1+t_2)$, \item[(3)]
$\overline{w} (M^{d\cdot 4})=1+(t_1+t_2+t_3)+t_1 t_3+t_1 t_2 t_3$,
\item[(4)] $\overline{w} (M^{d\cdot 5})=1+(t_1+t_2+t_3+t_4)+(t_1
t_3+t_1 t_4+t_2 t_4)+(t_1 t_2 t_3+t_2 t_3 t_4)$.
\end{itemize}
\end{example}

\subsection{The class $\overline{w}_{d (n-1)}(M^{dn})$}

Consider arbitrary two manifolds $M^{dn}$ and $M^{d(n+1)}$
constructed as in subsection \ref{cons} over simple polytopes
$P_1^{n}$ and $P_2^{n+1}$, properly colored in $n$ and $n+1$
colors, respectively.

\begin{lemma}\label{sub} The ring $\mathbb{Z}_2[t_1,\ldots,t_n]/\mathcal{T}_n$
is a subring of the cohomology ring $H^{\ast}(M^{dn};
\mathbb{Z}_2)$ which is generated by elements $t_1,\ldots,t_n$.
\end{lemma}
\begin{proof} It is necessary to prove that $\mathcal{T}_n$
is the ideal of all relations among elements $t_1,\ldots,t_n$ in
$H^{\ast}(M^{dn}; \mathbb{Z}_2)$. The rings $\mathbb{Z}_2 [t_1,
\dots, t_n]/\mathcal{T}_n$ and $\mathbb{Z}_2 [u_1, \dots,
u_n]/\mathcal{U}_n$ are isomorphic, where $\mathcal{U}_n$ is the
ideal generated by elements $u_1^{2}$ and
$u_i^{2}+(u_1+\ldots+u_{i-1})u_i, i=2,\ldots,n$ from Proposition
\ref{p2}. It is sufficient to show that $\mathcal{U}_n$ is the
ideal of all relations among $u_1,\ldots,u_n$ in $H^{\ast}(M^{dn};
\mathbb{Z}_2)$.

By Corollary \ref{posl1} any homogeneous polynomial in
$H^{\ast}(M^{dn}; \mathbb{Z}_2)$ is expressed as a sum of square
free monomials. So let we have a relation
$$ \sum_{j=1}^{m}u_{I_j}=0,$$ where
$I_j=\{a_{j1}<\ldots<a_{jk}\}\subset\{1,\ldots,n\}, j=1,\ldots,m$
and $u_{I}=\prod_{i\in I}u_i$. It is easy to convince that
$u_1\cdots u_{i-1}u_i^{2}=0, i=1,\ldots,n$. Let us order sets
lexicographically $I_1<\ldots<I_m$. Let $I_m=\{a_1<\ldots<a_k\}$
and $d_i=a_i-a_{i-1}, i=1,\ldots,k$ where $a_0=0$. Define
following elements
$$U_i=\left\{\begin{array}{cc} u_{a_{i-1}+1}\cdots u_{a_i-1}, & d_i>1
\\ 1, & d_i=1\end{array}\right.,$$ where $i=1,\ldots,k$. By multiplying
the relation with $U_1\cdots U_k,$ we get $u_1\cdots u_{a_k}=0,$
which contradicts the fact that
 $u_1\cdots u_n$ is the fundamental class.
\end{proof}

By abbreviating all terms in the identity in Lemma \ref{swd} we
obtain that each class $\overline{w}_{dk}(M^{dn})$ is expressed in
the ring $\mathbb{Z}_2[t_1,\dots,t_n]/\mathcal{T}_n$ by square
free homogeneous polynomials $\overline{W}_{k}(t_1,\dots,t_n)$ of
the degree $k$. Note that the inclusion
$$i:\mathbb{Z}_2[t_1, \dots, t_n]\rightarrow\mathbb{Z}_2[t_1,
\dots, t_n, t_{n+1}]$$ induces a natural monomorphism

\[i^*:\mathbb{Z}_2[t_1, \dots, t_n]/\mathcal{T}_n\rightarrow
\mathbb{Z}_2[t_1, \dots, t_n, t_{n+1}]/\mathcal{T}_{n+1},\] which
by Lemma \ref{sub} allows us to consider the total Stiefel-Whitney
class $\overline{w}(M^{dn})$ as an element of the ring
$H^*(M^{d(n+1)}; \mathbb{Z}_2)$. Thw total Stiefel-Whitney classes
$\overline{w} (M^{dn})$ and $\overline{w} (M^{d(n+1)})$ satisfy
the following relation in $H^*(M^{d(n+1)}; \mathbb{Z}_2)$
$$\overline{w} (M^{d(n+1)})=\overline{w}(M^{dn})(1+t_n+\dots+t_n^n).$$
Explicitly
\begin{equation}\label{jed2}
\overline{w}_{dk}(M^{d(n+1)})=\overline{w}_{dk}(M^{dn})+t_n\overline{w}_{d(k-1)}(M^{dn})+\dots+t_n^k,
k=0,\ldots, n.
\end{equation}
Recall that $\overline{w}_{dn}(M^{dn})=0$ (see \cite{10}), which
implies
\begin{equation}\label{jed3}
\overline{w}_{dn}(M^{d(n+1)})= t_n
\overline{w}_{d(n-1)}(M^{dn})+\cdots+t_n^n=t_n\overline{w}_{d(n-1)}(M^{d(n+1)}).
\end{equation}

We use the same trick as in \cite{1}. Define numbers $\sigma^k_n,
0\leq k\leq n-1$ as follows
\[
\sigma^k_n=\overline{W}_{dk}(\underbrace{1,\dots,1}_n)\pmod 2.
\]
By (\ref{jed2}) and (\ref{jed3}), we have
$\sigma^k_{n+1}=\sum_{i=0}^{k} \sigma^i_n $ for every
$k=1,\ldots,n-1$ and $\sigma^{n}_{n+1}=\sigma^{n-1}_{n+1}$. By
definition of $\sigma^k_n$, if $\sigma^k_n=1$, then
$\overline{w}_{dk}$ is the sum of an odd number of linearly
independent square free monomials, which implies
$\overline{w}_{dk}(M^{dn})\neq 0. $

An easy mathematical induction shows that
 $$ \sigma^k_n\equiv\bn{n+k}{k}\pmod{2}. $$
Particularly, if $n=2^r$, we have
\[
\sigma^{n-1}_n
\equiv\bn{2^r+(2^r-1)}{2^r-1}\equiv\bn{2^{r+1}-1}{2^r-1}\equiv
1\pmod{2}.
\]
Consequently,
\[
\overline{w}_{d(n-1)}(M^{dn})=t_1 t_2\cdots t_{n-1}\neq 0.
\]
Therefore by Theorem \ref{imeem} we obtain the required bounds
\[imm(M^{dn})\geq d(2n-1), em(M^{dn})\geq d(2n-1)+1.\] For the
small cover $M^{n}$, Whitney's theorem implies $imm(M^{n})=2n-1$
and $em(M^{n})=2n$. The quasitoric manifold $M^{2n}$ is
orientable, so it can be embedded into $\mathbb{R}^{4n-1}$. From
Lemma \ref{swd} follows
$\overline{w}_2(M^{2n})=t_1+t_2+\cdots+t_{n-1},$ which implies
that the characteristic class $\overline{w}_2(M^{dn})\cdot
\overline{w}_{2n-2}(M^{2n})$ vanishes. If $n\geq 3$, the result of
Massey \cite[Theorem V]{11} yields
 $$imm(M^{2n})=4n-2 ,$$ which finishes the proof of Theorem \ref{main}.

\section{Proof of Theorem \ref{exist}}

Let $n=2^{r_1}+2^{r_2}+\dots+2^{r_t}$, $r_1>r_2>\dots>r_t\geq 0$
be the binary representation of $n$ and let $m_i=2^{r_i}$ for
$i=1,\dots,t$ and $m_0=0$. Let $P^n$ be a simple $n$-polytope such
that $$P^n= P_1^{2_{r_1}}\times \cdots\times P_t^{2^{r_t}},
$$ where each $P_i^{2_{r_i}}$ is $2^{r_i}$-colored simple
$2^{r_i}$-polytope. It is obvious that polytope $P^n$ is
$n$-colored.

In subsection \ref{cons} we constructed manifolds $M^{d 2^{r_i}}$
over polytopes $P_i^{2_{r_i}}$. It follows from
\cite[Proposition~4.7]{2} that
$M^{dn}=M^{d2^{r_1}}\times\dots\times M^{d 2^{r_t}}$ is a
$G_d^{n}$-manifold over the polytope $P^n= P_1^{2_{r_1}}\times
\cdots\times P_t^{2^{r_t}}$. The total Stiefel--Whitney class of
the manifold $M^{dn}$ can be easily determined using the following
formula (see \cite[pp.\,27,\,54]{4})
\[
w(M^{dn})=w(M^{d 2^{r_1}})\cdots w(M^{d 2^{r_t}})\in
H^*(M^{dn})\cong H^*(M^{d 2^{r_1}})\otimes\dots\otimes H^*(M^{d
2^{r_t}}).
\]
The corresponding dual total Stiefel-Whitney class is expressed as
\begin{equation}\label{dpr} \overline{w}(M^{dn})=\overline{w}(M^{d 2^{r_1}})\cdots
\overline{w}(M^{d 2^{r_t}}).\end{equation} Let $\rank w M):=\max
\{k | w_k (M)\neq 0\}$. Thus, from formula (\ref{dpr}) we have

$$\mathrm{rank} \ \overline{w}(M^{dn})=\sum_{i=1}^t\mathrm{rank} \ \overline{w}(M^{d2^{r_i}})=\sum_{i=1}^t d(2^{r_i}-1)=nd -\alpha (n) d.$$
In this way, Theorem \ref{exist} is a consequence of Theorems
\ref{imeem} and \ref{skewteo}.

\section*{Aknowledgements}

The authors express their special thanks to T. E. Panov and A. A.
Gaifullin for attention to this work, as well as to reviewers for
helpful comments.

\end{document}